\documentclass[10pt]{amsart}
\usepackage{amssymb}
\usepackage{amsmath}
\usepackage{amsthm}
\usepackage[all]{xy}
\usepackage{mathrsfs}
\usepackage{array}
\usepackage{caption}
\usepackage{lscape}
\usepackage{color}
\usepackage{pstricks}
\usepackage{graphicx}

\include{scrload}

\newtheorem{thm}{Theorem}[section]

\newtheorem{prop}[thm]{Proposition}
\newtheorem{lemma}[thm]{Lemma}
\newtheorem{cor}[thm]{Corollary}

\theoremstyle{definition}
\newtheorem{rmk}[thm]{Remark}
\newtheorem{eg}[thm]{Example}
\newtheorem{defn}[thm]{Definition}

\newenvironment{pf}{\begin{proof}}{\end{proof}}

\newcommand{\C}{\ensuremath{\mathbb{C}}}

\newcommand{\F}{\ensuremath{\mathbb{F}}}

\newcommand{\bM}{\ensuremath{\mathbb{M}}}

\newcommand{\rtarr}{\longrightarrow}
\newcommand{\ltarr}{\longleftarrow}
\newcommand{\xrtarr}[1]{\xrightarrow{#1}}
\newcommand{\xltarr}[1]{\xleftarrow{#1}}
\newcommand{\iso}{\cong}

\newcommand{\Sq}{\ensuremath{\mathrm{Sq}}}

\DeclareMathOperator{\Hom}{Hom}

\DeclareMathOperator{\Ext}{Ext}

\newcommand{\cirrad}{0.11}
\newgray{gridline}{0.6}

\newcolumntype{C}{>{$}c<{$}}
\newcolumntype{L}{>{$}l<{$}}
\newcolumntype{R}{>{$}r<{$}}
\newcolumntype{H}{>{\setbox0=\hbox\bgroup$}c<{$\egroup}@{}}

\begin{document}
\title{The motivic Adams vanishing line of slope $1/2$}
\author{Bertrand J. Guillou}
\address{Department of Mathematics\\ University of Kentucky\\
Lexington, KY 40506, USA}
\email{bertguillou@uky.edu}

\author{Daniel C. Isaksen}
\address{Department of Mathematics\\ Wayne State University\\
Detroit, MI 48202, USA}
\email{isaksen@wayne.edu}
\thanks{The first author was supported by Simons Collaboration Grant 282316.
The second author was supported by NSF grant DMS-1202213.}



\begin{abstract}
We establish a motivic version of Adams' vanishing line 
of slope $1/2$ in the cohomology of the motivic Steenrod algebra
over $\C$.
\end{abstract}

\date{\today}

\maketitle

\section{Introduction}

One of the most powerful tools for computing stable homotopy groups of spheres is the Adams spectral sequence
\[ \Ext^{s,t}_A(\F_2,\F_2) \Rightarrow \hat\pi_{t-s}, \]
where $A$ denotes the Steenrod algebra of 
stable mod $2$ cohomology operations and $\hat\pi_*$ are the 2-completed stable homotopy groups.
Adams established that these $\Ext$ groups vanish above a certain line of slope $1/2$, with the exception of the elements $h_0^k$ in the $0$-stem
\cite{A2}.

In the motivic context over $\C$,
inspection of an Adams chart \cite{Icharts}
immediately shows that the analogous $\Ext$ groups do not vanish
above the same line of slope $1/2$.  (However, the motivic
$\Ext$ groups do vanish above a line of slope $1$, analogously to
an earlier classical vanishing result of Adams \cite{Adams1}.)
Further inspection shows that in a large range, all elements
above the Adams line of slope $1/2$ are $h_1$-local, in the sense
that they are $h_1$-divisible and support infinitely many multiplications
by $h_1$.  This suggests the following theorem, whose proof is the goal
of this article.

\begin{thm}\label{MainThm} Let $s>0$, and let
$A$ be the motivic mod 2 Steenrod algebra over $\C$.
The map
\[ 
h_1:\Ext_A^{s,f,w}(\bM_2,\bM_2) \rightarrow \Ext_A^{s+1,f+1,w+1}(\bM_2,\bM_2)
\]
is an isomorphism 
if $f \geq \frac12s+2$, and it is a surjection if $f\geq \frac12s+\frac12$.
\end{thm}

In the grading $(s,f,w)$ from Theorem \ref{MainThm},
$s$ is the stem and $f$ is the Adams filtration so that
$(s,f)$ represent the traditional coordinates in an Adams chart.
Meanwhile, $w$ is the motivic weight, which is not relevant to the 
statement of the theorem.

\begin{pf}
In Section \ref{sec:CohomA2}, we will show that 
\[\Ext_A^{s,f,w}(\bM_2,\bM_2) \iso \Ext_A^{s-1,f-1,w-1}(N,\bM_2),\]
where $N$ is a certain $A$-module that is free as a left $A(0)$-module.
We will prove in Proposition \ref{MainProp} that
multiplication by $h_1$ has the desired property for such $A$-modules.
\end{pf}

Inspection shows that 
Theorem \ref{MainThm} is optimal in the following sense.
Multiplication by $h_1$ is not an isomorphism
along the line $f = \frac{1}{2} s + \frac{3}{2}$,
and it is not a surjection along the line
$f = \frac{1}{2} s$.

This article is only concerned with algebraic calculations of $\Ext$ groups
and does not discuss Adams spectral sequences for which these $\Ext$ groups
are inputs.  However, our interest in $h_1$ multiplications is motivated by
recent work \cite{AM2} \cite{GI}
on the $\eta$-local motivic sphere, where $\eta$ is the first
Hopf map that is detected by $h_1$.

Another possible approach to Theorem \ref{MainThm} uses
the methods established by Andrews and Miller \cite{AM2}
(see especially \cite[Lemma 5.2.1]{AM2}).
However, one must be careful about the slightly strange behavior
of motivic Margolis homology (see Proposition \ref{prop:FreeA0}).
For the techniques of \cite{AM2} to work motivically, one needs
to verify that certain quotients of the motivic Steenrod algebra
are $\bM_2$-free.  This is more or less the same as the 
homological results of Sections \ref{sctn:A0-mod} and \ref{sctn:Milnor}.

\subsection{Organization}
We review some background in Section \ref{sec:background}.
In Sections \ref{sctn:A0-mod} through \ref{sctn:cofiber-eta},
we assemble various facts about motivic homological algebra.
Finally, 
we give the main technical result in Proposition \ref{MainProp},
from which Theorem \ref{MainThm} follows immediately.


\section{Preliminaries}\label{sec:background}

\subsection{Notation}

We continue with notation from \cite{Istems} as follows:
\begin{enumerate}
\item
$\bM_2$ 
is the motivic cohomology of $\C$ with $\F_2$ coefficients.
\item
$A$ is the mod 2 motivic Steenrod algebra over $\C$ generated by elements $\Sq^1$ of bidegree $(1,0)$ and $\Sq^{2^n}$ of bidegree $(2^n,2^{n-1})$ for $n\geq 1$.
\item
$A(0)$ is the (exterior) $\bM_2$-subalgebra of $A$ generated by $\Sq^1$.
\item
$A(1)$ is the $\bM_2$-subalgebra of $A$ generated by
$\Sq^1$ and $\Sq^2$.
\item
$\Ext_A$ is the trigraded ring $\Ext_A(\bM_2,\bM_2)$.
\end{enumerate}

The following two  theorems of Voevodsky are the starting points
of our calculations.

\begin{thm}[\cite{V1}]
$\bM_2$ is the bigraded ring $\F_2[\tau]$, where
$\tau$ has bidegree $(0,1)$.
\end{thm}

It is often more convenient to work with the dual $A_{*,*} = \Hom_{\bM_2}(A,\bM_2)$, which was described by Voevodsky. (See also \cite{Borghesi} for a clean description.)

\begin{thm}\label{thm:StructureDual}\cite{V2} \cite[Theorem 12.6]{V3} 
The dual motivic Steenrod algebra $A_{*,*}$ is generated as an $\bM_2$-algebra by 
$\xi_i\in A_{2(2^i-1),2^i-1}$ and $\tau_i\in A_{2^{i+1}-1,2^i-1}$ subject to the
relations
\[ \tau_i^2 = \tau \xi_{i+1}.\]
The coproduct is given on the generators by the following formulas, in which
$\xi_0 = 1$:
\[ \Delta(\tau_k) = \tau_k\otimes 1 + \sum_i \xi_{k-i}^{2^i}\otimes \tau_i\]
\[ \Delta(\xi_k) = \sum_i  \xi_{k-i}^{2^i} \otimes \xi_i.\]
\end{thm}

\begin{rmk}
The quotient $A_{*,*}/\tau=A_{*,*}\otimes_{\bM_2}\F_2$ is analogous to the 
odd-primary classical dual Steenrod algebra, in the sense that there is an
infinite family of exterior generators $\tau_i$ and an infinite family
of polynomial generators $\xi_i$.  On the other hand,
the localization $A_{*,*}[\tau^{-1}]$ is analogous to the
mod 2 classical dual Steenrod algebra, which has only polynomial generators
$\tau_i$.
\end{rmk}

\subsection{Grading conventions}

We follow \cite{Istems} in grading $\Ext_A$
according to  $(s,f,w)$, where:
\begin{enumerate}
\item
$f$ is the Adams filtration, i.e., the homological degree.
\item
$s+f$ is the internal degree, i.e., corresponds to the
first coordinate in the bidegrees of $A$.
\item
$s$ is the stem, i.e., the internal degree minus
the Adams filtration.
\item
$w$ is the weight.
\end{enumerate}

\section{Margolis homology}
\label{sctn:A0-mod}

Recall that $\Sq^1 \Sq^1=0$, so that $\Sq^1$ acts as a differential on any $A$-module. 
We write $H^{**}(M;\Sq^1)$ for the resulting 
{\bf Margolis homology} groups \cite{AM1}.
We say that an $A$-module is {\bf bounded below} if $M^{p,q}=0$ 
for sufficiently small $p$.
A bounded below $A$-module is {\bf of finite type} if each
$M^{p,*}$ is a finitely generated $\bM_2$-module.

\begin{rmk}
\label{rmk:M2-module}
We will need the following fact about finitely generated
$\bM_2$-modules.  Such modules are of the form
\[
(\bM_2)^k \oplus \bigoplus_i (\bM_2 / \tau^{r_i} ).
\]
This is a graded version of the classification of finitely generated
modules over a principal ideal domain, since
every
homogeneous ideal of $\bM_2$ is generated
by an element of the form $\tau^{r_i}$.
Consequently,
a finitely generated $\bM_2$-module is free if and only if 
it has no $\tau$ torsion.
\end{rmk}

\begin{prop}\label{prop:FreeA0} Let $M$ be a bounded below $A$-module
of finite type. 
Then $M$ is free as an $A(0)$-module if and only if $M$ is free as an $\bM_2$-module and $H^{*,*}(M;\Sq^1)=0$.
\end{prop}

\begin{pf} The forward implication is clear. Thus suppose that $M$ is free over
$\bM_2$ and $H^{*,*}(M;\Sq^1)=0$. Suppose that $M$ is concentrated in 
degrees $(p,q)$ with $p \geq n_0$.
Let $x$ be a nonzero element of $M^{n_0,*}$ of smallest weight, and let
$y = \Sq^1(x)$ in $M^{n_0+1,*}$.
The map $\Sq^1: M^{n_0,*} \to M^{n_0+1,*}$
is injective because $H^{*,*}(M;\Sq^1) = 0$,
so $y$ is non-zero.

Let $N$ be the $\bM_2$-submodule of $M$ generated by
$x$ and $y$, and let $P$ be the quotient
$M/N$.  Then $N$ is a free $A(0)$-module
generated by $x$.

We will next argue that $P$ is $\bM_2$-free;
by Remark \ref{rmk:M2-module}, this is the same as showing
that $P$ has no $\tau$ torsion.
Equivalently, we will show that if 
$\tau z$ belongs to $N$, then so does $z$.
The main point is that $y$ is not divisible by $\tau$.
Suppose for sake of contradiction that $\tau z = y$.
Then $\tau \Sq^1(z) = \Sq^1(y) = 0$.
Since $M$ is $\bM_2$-free,
Remark \ref{rmk:M2-module} implies that
$\Sq^1(z) = 0$.
On the other hand, $z$ cannot be in the image of $\Sq^1$
because its weight is less than the weight of $x$.
This contradicts the assumption that 
$H^{*,*}(M;\Sq^1)$ is zero.

This establishes that
$M$ is isomorphic to $N \oplus P$ as an $A(0)$-module.
The long exact sequence associated to the short exact sequence
\[
0 \rtarr N \rtarr M \rtarr P \rtarr 0
\]
shows that $H^{*,*}(P;\Sq^1) = 0$.

Having split $M$ as $N \oplus P$, we may
now apply the same argument to $P$ to split
off another free $A(0)$-module.  The finite type assumption on $M$
guarantees that this process eventually splits all of $M$ 
into free $A(0)$-modules.
\end{pf}

\begin{cor}\label{cor:A0FreeSOS}
Suppose that 
\[ 0 \rtarr M_1 \rtarr M_2 \rtarr M_3 \rtarr 0 \]
is a short exact sequence of bounded below $A$-modules 
of finite type
which are free over $\bM_2$. If any two modules in this sequence are free over $A(0)$, then so is the third.
\end{cor}

\begin{proof}
Two of the modules have no Margolis homology
by Proposition \ref{prop:FreeA0}.
The long exact sequence in Margolis homology shows that the
third also has no Margolis homology.
Then Proposition \ref{prop:FreeA0} again implies that the third
module is free over $A(0)$.
\end{proof}

\section{The motivic Milnor basis}
\label{sctn:Milnor}

Let $E=(\epsilon_0,\epsilon_1,\epsilon_2,\dots)$ be a sequence of ones and zeros, almost all zero, and let $R=(r_1,r_2,\dots)$ be a sequence of nonnegative integers, almost all zero. Then, according to Theorem~\ref{thm:StructureDual}, the elements 
\[ \tau(E)\xi(R) := \prod_{i\geq 0} \tau_i^{\epsilon_i} \prod_{j\geq 1} \xi_j^{r_j}\]
give an $\bM_2$-basis for $A_{*,*}$. 
We follow \cite{DI} in writing $P^{(\epsilon_0+2r_1,\epsilon_1+2r_2,\dots)}$ for the corresponding elements of the dual $\bM_2$-basis for $A$. 
These elements form the {\bf motivic Milnor basis} for $A$.
The resulting 
$\bM_2$-basis for $A(1)$ consists of the elements $P^{(s_1,s_2,0,0\dots)}$ 
such that $0\leq s_1\leq 3$ and $0\leq s_2\leq 1$.

Let $S = (s_1, s_2, \ldots )$.
The {\bf excess} of the Milnor basis element $P^S$ is defined to be 
$e(P^S) = \sum_i s_i$. We extend this to arbitrary elements of $A$ by taking the excess of an element $\theta$ to be the maximal excess of any Milnor basis element appearing in an expression for $\theta$.  

For two sequences $R = (r_1, r_2, \ldots)$ and $S = (s_1, s_2, \ldots)$,
we write $R+S$ for the termwise sum $(r_1 + s_1, r_2 + s_2, \ldots)$.

\begin{lemma}
\label{lem:Milnor-product}
Let $R = (r_1, r_2, \ldots)$ and $S = (s_1, s_2, \ldots)$.
Then
\[P^R \cdot P^S  = \prod_i \binom{r_i}{s_i} \cdot P^{R+S} \ + 
 \emph{ terms of lower excess.} \]
\end{lemma}

\begin{pf}
This follows from the description of the dual motivic Steenrod algebra,
similarly to the classical case \cite{Milnor}.
\end{pf}

We will use the Milnor basis to establish the following fact about
the right action of $A(1)$ on $A$.

\begin{prop}\label{prop:AFreeA1} 
The motivic Steenrod algebra $A$ is free as a right $A(1)$-module.
\end{prop}

\begin{pf}
Define a filtration on $A$ by $F_s(A) = \{ \theta\in A \mid e(\theta) \leq s\}$. 
Lemma \ref{lem:Milnor-product}
implies that the associated graded object
\[ \mathrm{gr_*} A = \bigoplus_s F_s(A)/F_{s-1}(A)\]
inherits the structure of an $\bM_2$-algebra. 
Let $M$ be the $\bM_2$-submodule of $A$ generated by Milnor
basis elements of the form $P^{(4s_1,2s_2,s_3,s_4,\dots)}$.
Consider the composition
\[ \xymatrix{
M \otimes_{\bM_2} A(1) \ar[r] & A \ar[r] & \mathrm{gr_*}A,
}
\]
in which the first map is induced by the multiplication of $A$
and the second map is an isomorphism of $\bM_2$-modules.

Lemma \ref{lem:Milnor-product} implies that the composition is an isomorphism
of $\bM_2$-modules.  Therefore, the first map is also an isomorphism.
This shows that $A$ is free as a right $A(1)$-module with basis consisting
of elements of the form $P^{(4s_1,2s_2,s_3,s_4,\dots)}$.
\end{pf}

\begin{rmk}
The reader may wonder why Proposition \ref{prop:AFreeA1} is
not an immediate consequence of results in \cite{MM}.
The problem is that it is not obvious that the projection $A \rtarr A/\!/A(1)$ is
split as an $\bM_2$-map.  The proof of Proposition \ref{prop:AFreeA1}
is essentially the same as showing that the projection is in fact split.
\end{rmk}

\begin{eg}
Let $B$ be the Hopf subalgebra of $A$ that is generated over $\bM_2$ by
$\tau \Sq^1$.  The projection $A \rtarr A/\!/B$ is not split as an $\bM_2$-map,
since $\Sq^1$ projects to an element that is $\tau$ torsion.
Moreover, $A$ is not free as a right $B$-module.
\end{eg}

\section{The module $\widetilde{A}$.}

\begin{defn} 
Let $\widetilde{A}(1)$ be the left $A(1)$-module
on two generators $a$ and $b$ of degrees $(0,0)$ and $(2,0)$ respectively,
subject to the relations
\[ \Sq^2 a = \tau b,\quad  \Sq^1 \Sq^2 \Sq^1 a = \Sq^2 b. \]
\end{defn}

The relation $\Sq^2 a = \tau b$ implies that
$\Sq^2 \Sq^2 a = \tau \Sq^2 b$, so
$\tau \Sq^1 \Sq^2 \Sq^1 a = \tau \Sq^2 b$.
However, the first relation does not imply that
$\Sq^1 \Sq^2 \Sq^1 a = \Sq^2 b$.  This explains why we need 
a second relation in the definition of $\widetilde{A}(1)$.

Figure~\ref{fig:Atilde} represents $\widetilde{A}(1)$, according to the
following key:
\begin{enumerate}
\item
Each circle represents a copy of $\bM_2$.
\item  
Each straight line represents multiplication by $\Sq^1$.
\item
Each curved line represents multiplication by $\Sq^2$.
\item
Each dashed line indicates that the squaring operation hits 
$\tau$ times an $\bM_2$-generator, but not the generator itself.
\end{enumerate}

\begin{figure}[h]

\psset{unit=7mm}

\begin{pspicture}(-1,0)(1,6)

\psline(0,0)(0,1)
\psline(0,2)(1,3)
\psline(-1,3)(0,4)
\psline(0,5)(0,6)

\psbezier[linestyle=dashed](0,0)(0.7,0.7)(0.7,1.3)(0,2)
\psbezier(0,1)(-0.7,1.7)(-1,2.3)(-1,3)
\psbezier(0,2)(0.5,2.7)(0.5,3.3)(0,4)
\psbezier(1,3)(1,3.7)(0.7,4.3)(0,5)
\psbezier[linestyle=dashed](0,4)(-0.7,4.7)(-0.7,5.3)(0,6)

\pscircle[fillstyle=solid,fillcolor=white](0,0){0.2}
\pscircle[fillstyle=solid,fillcolor=white](0,1){0.2}
\pscircle[fillstyle=solid,fillcolor=white](0,2){0.2}
\pscircle[fillstyle=solid,fillcolor=white](-1,3){0.2}
\pscircle[fillstyle=solid,fillcolor=white](1,3){0.2}
\pscircle[fillstyle=solid,fillcolor=white](0,4){0.2}
\pscircle[fillstyle=solid,fillcolor=white](0,5){0.2}
\pscircle[fillstyle=solid,fillcolor=white](0,6){0.2}

\end{pspicture}

\caption{{$\widetilde{A}(1)$} }
\label{fig:Atilde}
\end{figure}

Analogous to the Milnor basis for $A(1)$,
we have a basis for $\widetilde{A}(1)$ 
consisting of the elements 
$\widetilde{P}^{(s_1, s_2)}$ such that 
$0 \leq s_1 \leq 3$ and $0 \leq s_2 \leq 1$.
The elements $\widetilde{P}^2$, $\widetilde{P}^3$, $\widetilde{P}^{2,1}$, and 
$\widetilde{P}^{3,1}$ have weight one less than the corresponding Milnor basis elements for $A(1)$. We define the excess in $\widetilde{A}(1)$ using this basis.


\begin{rmk}
Just like in the classical case \cite{DM},
the $A(1)$-modules $A(1)$ and $\widetilde{A}(1)$ each extend to $A$-modules
in four different ways, 
determined by the action of $\Sq^4$.
Adams spectral sequence computations verify that 
all eight of these $A$-modules arise as the cohomology of a $2$-complete motivic spectrum. 
These constructions are the subject of work in progress on 
motivic $v_1$-self maps.
\end{rmk}

\begin{defn}
Let $\widetilde{A}$ be the left $A$-module
$A\otimes_{A(1)} \widetilde{A}(1)$.
\end{defn}

\begin{rmk}
Proposition~\ref{prop:AFreeA1} implies that 
$\widetilde{A}$ is the left $A$-module generated by two elements
$a$ and $b$ of degrees $(0,0)$ and $(2,0)$ subject to the relations
\[\Sq^2 a=\tau b, \quad \Sq^3\Sq^1 a = \Sq^2 b.\]
\end{rmk}

\begin{prop}\label{prop:AtildeA0Free}
The $A$-module $\widetilde{A}$ is free as a left $A(0)$-module.
\end{prop}

\begin{pf}
We have a Milnor-style basis for $\widetilde{A}$ consisting of elements of the form
$P^{(4r_1,2r_2,r_3,\dots)} \otimes \widetilde{P}^{s_1,s_2}$ 
such that $0 \leq s_1 \leq 3$ and $0 \leq s_2 \leq 1$.
We also have a filtration by excess.

Lemma \ref{lem:Milnor-product} implies that 
\[ P^1 P^{(4r_1,2r_2,r_3,\dots)} \equiv P^{(4r_1+1,2r_2,r_3,\dots)} \equiv P^{(4r_1,2r_2,r_3,\dots)}P^1\]
in $A$, modulo terms of lower excess.
Also,
$P^1\widetilde{P}^{s_1,s_2} = \widetilde{P}^{s_1+1,s_2}$ in $\widetilde{A}(1)$
if $s_1$ is even.
Therefore, if $s_1$ is even, then
\[ P^1 (P^{(4r_1,2r_2,r_3,\dots)}\otimes \widetilde{P}^{s_1,s_2}) \equiv P^{(4r_1,2r_2,r_3,\dots)}\otimes \widetilde{P}^{s_1+1,s_2} 
\]
in $\widetilde{A}$, modulo terms of lower excess.
It follows that 
an $A(0)$-basis for $\widetilde{A}$ consists of elements of the form
$P^{(4r_1,2r_2,r_3,\dots)} \otimes \widetilde{P}^{s_1,s_2}$ 
such that $s_1$ is even.
\end{pf}

\section{The cofiber of $\eta$} 
\label{sctn:cofiber-eta}

\begin{defn} Let $C(\eta)$ denote the cofiber of 
the first Hopf map $\eta:S^{1,1}\rtarr S^{0,0}$. We then write $\mathbf{C}_\eta$ for the $A$-module $\Sigma^{-2,-1} \mathrm{H}^{*,*}(C(\eta))$. Thus $\mathbf{C}_\eta$ has a bottom cell in bidegrees $(-2,-1)$ and a top cell in $(0,0)$, connected by a $\Sq^2$ (see figure).
\end{defn}

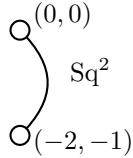
\psset{unit=7mm}
\begin{center}
\begin{pspicture}(-1,-1)(1,2.5)
\psbezier(0,0)(0.7,0.7)(0.7,1.3)(0,2)
\pscircle[fillstyle=solid](0,0){0.2}
\pscircle[fillstyle=solid](0,2){0.2}
\uput[30](0,2){$(0,0)$}
\uput[-15](0,0){$(-2,-1)$}
\uput[0](0.7,1.2){$\Sq^2$}
\end{pspicture}
\end{center}

The following result implies that, for any $A$-module $M$, the groups $\Ext^*(M,\mathbf{C}_\eta)$ may be computed using a resolution $F^*\to M$ whose terms are of the form \mbox{$F^n \iso A^r \oplus \widetilde{A}^s$.}

\begin{prop} 
\label{prop:Ceta-proj}
$\Ext^{s,f}_A(\widetilde{A},\mathbf{C}_\eta)=0$ for $f>0$.
\end{prop}

\begin{pf}
The kernel of the defining quotient $A(1) \oplus \Sigma^{2,0} A(1) \rtarr \widetilde{A}(1)$ is isomorphic to $\Sigma^{2,1}\widetilde{A}(1)$. It follows that we can define a periodic free $A(1)$-resolution
\[ \widetilde{A}(1) \leftarrow F_0 \leftarrow F_1\leftarrow \cdots,\]
where $F_n \iso A(1)\{x_n,y_n\}$, with $|x_n| = (2n,n)$,  $|y_n| = (2n+2,n)$, 
\[d(x_n) = \Sq^2(x_{n-1}) + \tau\cdot y_{n-1},\qquad  \text{and} \qquad d(y_n) = \Sq^3\Sq^1(x_{n-1}) + \Sq^2(y_{n-1}).\]
Using this resolution, it is simple to verify that 
\[ \Ext_{A(1)}^{s,f}(\widetilde{A}(1),\mathbf{C}_\eta) = 0, \qquad \text{for } f>0.\]
Since $A$ is free as a right $A(1)$-module by Proposition~\ref{prop:AFreeA1}, we get that
\[ \Ext_A^{s,f}(\widetilde{A},\mathbf{C}_\eta) \iso \Ext_{A(1)}^{s,f}(\widetilde{A}(1),\mathbf{C}_\eta) = 0 \qquad \text{for } f>0. \]
\end{pf}

The next lemma translates properties of $h_1$ multiplication into homological
properties of the $A$-module $\mathbf{C}_\eta$.

\begin{lemma}
\label{lem:RestateLoc} 
Let $n\geq 0$ and let $M$ be an $A$-module.
The following are equivalent:
\begin{enumerate}
\item
The map
\[ h_1:\Ext_A^{s,f,w}(M,\bM_2) \xrightarrow{} \Ext_A^{s+1,f+1,w+1}(M,\bM_2) \]
is injective when $s<2f-2$ and $f\leq n$, and 
it is surjective when $s\leq 2f$ and $f<n$.
\item
$\Ext_A^{s,f,w}(M,\mathbf{C}_\eta)=0$ when $s<2f$ and $f\leq n$.
\end{enumerate}
\end{lemma}

\begin{pf}
The $A$-module $\mathbf{C}_\eta$ sits in a short exact sequence
\[ 0 \rtarr \bM_2 \rtarr \mathbf{C}_\eta \rtarr \Sigma^{-2,-1}\bM_2 \rtarr 0.
\]
This gives a long exact sequence 
\[
\xrightarrow{h_1}
\Ext_A^{s+1,f-1,w}(M,\bM_2) \to 
\Ext_A^{s+1,f-1,w}(M,\mathbf{C}_\eta) \to
\Ext_A^{s-1,f-1,w-1}(M,\bM_2) \xrightarrow{h_1} 
\]
in which
the connecting homomorphism is multiplication by $h_1$. 
The result then follows easily.
\end{pf}

Our next goal is to carry out an explicit low-dimensional $\Ext$ computation.
This result will be critical for a later argument.
We will consider $A(0)$ to be an $A$-module with the obvious action by $A$,
i.e., $\Sq^1$ acts by multiplication in the subalgebra $A(0)$, and
$\Sq^n$ acts trivially for $n \geq 2$.

\begin{prop}
\label{prop:low-Ext}
$\Ext^{s,f,w}(A(0),\mathbf{C}_\eta)=0$ when $s<2f$ and $f\leq 4$.
\end{prop}

\begin{pf}
The short exact sequence
\[ 0 \rtarr \Sigma^{1,0}\bM_2 \rtarr A(0) \rtarr \bM_2 \rtarr 0 \]
gives a calculation of  $\Ext(A(0),\bM_2)$ in low degrees, starting from
knowledge of $\Ext(\bM_2,\bM_2)$ in a similar range. 
The short exact sequence 
\[ 0 \rtarr \bM_2 \rtarr \mathbf{C}_\eta \rtarr \Sigma^{-2,-1}\bM_2
\rtarr 0
\]
then yields a calculation of $\Ext(A(0),\mathbf{C}_\eta)$.
The results of these calculations are shown in 
Figures~\ref{fig:A(0)M} and \ref{fig:A(0)C}.
\end{pf}

Figures \ref{fig:A(0)M} and \ref{fig:A(0)C}
represent low-dimensional $\Ext$ calculations necessary
for the proof of Proposition \ref{prop:low-Ext}.
Here is a key for reading the charts.
\begin{enumerate}
\item
Black dots indicate copies of $\bM_2$.
\item
Vertical lines indicate multiplications by $h_0$.
\item
Lines of slope $1$ indicate multiplications by $h_1$.
\item
Lines of slope $1/3$ indicate multiplications by $h_2$.
\item
Blue lines indicate that the multiplication hits $\tau$ times a generator.
\item
Red arrows indicate infinitely many copies of $\bM_2/\tau$ connected by
$h_1$ multiplications.
\end{enumerate}

\psset{linewidth=0.3mm}    
\psset{unit=5mm}
\begin{figure}[h]
\begin{pspicture}(-1,-1)(13,8)

\psgrid[unit=1,gridcolor=gridline,subgriddiv=0,gridlabelcolor=white](-0.5,-0.5)(-0.5,-0.5)(13,8)

\scriptsize 
\psline(0.5,0.5)(2.5,2.5)
\psline[linecolor=red]{->}(2.5,2.5)(3.2,3.2)
\psline(2.5,1.5)(4.5,3.5)
\psline[linecolor=red]{->}(4.5,3.5)(5.2,4.2)
\psline[linecolor=blue](2.5,1.5)(2.5,2.5)
\psline(0.5,0.5)(9.3,3.5)
\psline(7.5,1.5)(9.3,3.5)
\psline(7.5,2.5)(9.5,4.5)
\psline[linecolor=red]{->}(9.5,4.5)(10.2,5.2)
\psline(7.5,2.5)(10.5,3.5)
\psline(9.5,2.5)(10.5,3.5)
\psline(8.5,4.5)(10.5,6.5)
\psline(8.5,4.5)(11.5,5.5)
\psline[linecolor=red]{->}(10.5,6.5)(11.2,7.2)
\psline(9.7,3.5)(10.6,4.5)
\psline[linecolor=red]{->}(10.6,4.5)(11.2,5.2)
\psline[linecolor=blue](9.7,3.5)(9.5,4.5)
\psline(10.5,5.5)(12.5,7.5)
\psline[linecolor=red]{->}(12.5,7.5)(13.2,8.2)
\psline[linecolor=blue](10.5,5.5)(10.5,6.5)
\psline[linecolor=blue](9.5,2.5)(9.3,3.5)

\pscircle*(0.5,0.5){\cirrad}
\pscircle*(1.5,1.5){\cirrad}
\pscircle*(2.5,2.5){\cirrad}
\pscircle*(2.5,1.5){\cirrad}
\pscircle*(3.5,2.5){\cirrad}
\pscircle*(4.5,3.5){\cirrad}
\pscircle*(3.5,1.5){\cirrad}
\pscircle*(6.5,2.5){\cirrad}
\pscircle*(9.3,3.5){\cirrad}
\pscircle*(7.5,1.5){\cirrad}
\pscircle*(8.4,2.5){\cirrad}
\pscircle*(8.5,3.5){\cirrad}
\pscircle*(9.5,4.5){\cirrad}
\pscircle*(7.5,2.5){\cirrad}
\pscircle*(10.5,3.5){\cirrad}
\pscircle*(9.5,2.5){\cirrad}
\pscircle*(8.5,4.5){\cirrad}
\pscircle*(9.5,5.5){\cirrad}
\pscircle*(10.5,6.5){\cirrad}
\pscircle*(9.7,3.5){\cirrad}
\pscircle*(10.6,4.5){\cirrad}
\pscircle*(11.5,5.5){\cirrad}
\pscircle*(10.5,5.5){\cirrad}
\pscircle*(11.5,6.5){\cirrad}
\pscircle*(12.5,7.5){\cirrad}

\rput(0.5,-0.5){0}
\rput(2.5,-0.5){2}
\rput(4.5,-0.5){4}
\rput(6.5,-0.5){6}
\rput(8.5,-0.5){8}
\rput(10.5,-0.5){10}
\rput(12.5,-0.5){12}
\psline{->}(11,-1.2)(12.5,-1.2)
\rput(11.5,-1.6){$s$}
\rput(-0.5,0.5){0}
\rput(-0.5,2.5){2}
\rput(-0.5,4.5){4}
\rput(-0.5,6.5){6}
\psline{->}(-1.0,6)(-1.0,7.5)
\rput(-1.5,6.7){$f$}

\end{pspicture}

\caption{
{$\Ext_A(A(0),\bM_2)$} }
\label{fig:A(0)M}
\end{figure}

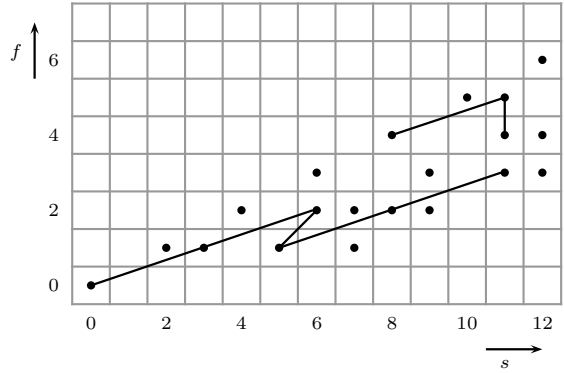
\begin{figure}[h]
\begin{pspicture}(-1,-1)(13,8)

\psgrid[unit=1,gridcolor=gridline,subgriddiv=0,gridlabelcolor=white](-0.5,-0.5)(-0.5,-0.5)(13,8)

\scriptsize 
\pscircle*(0.5,0.5){\cirrad}
\pscircle*(2.5,1.5){\cirrad}
\pscircle*(4.5,2.5){\cirrad}
\pscircle*(6.5,3.5){\cirrad}
\pscircle*(3.5,1.5){\cirrad}
\pscircle*(6.5,2.5){\cirrad}
\psline(0.5,0.5)(6.4,2.5)
\pscircle*(7.5,1.5){\cirrad}
\pscircle*(7.5,2.5){\cirrad}
\pscircle*(9.5,2.5){\cirrad}
\pscircle*(8.5,4.5){\cirrad}
\pscircle*(9.5,3.5){\cirrad}
\pscircle*(11.5,5.5){\cirrad}
\pscircle*(10.5,5.5){\cirrad}
\pscircle*(5.5,1.5){\cirrad}
\pscircle*(8.5,2.5){\cirrad}
\psline(8.5,4.5)(11.5,5.5)
\pscircle*(11.5,3.5){\cirrad}
\psline(5.5,1.5)(11.4,3.5)
\psline(5.5,1.5)(6.5,2.5)
\pscircle*(11.5,4.5){\cirrad}
\psline(11.5,4.5)(11.5,5.5)
\pscircle*(12.5,4.5){\cirrad}
\pscircle*(12.5,3.5){\cirrad}
\pscircle*(12.5,6.5){\cirrad}

\rput(0.5,-0.5){0}
\rput(2.5,-0.5){2}
\rput(4.5,-0.5){4}
\rput(6.5,-0.5){6}
\rput(8.5,-0.5){8}
\rput(10.5,-0.5){10}
\rput(12.5,-0.5){12}
\psline{->}(11,-1.2)(12.5,-1.2)
\rput(11.5,-1.6){$s$}
\rput(-0.5,0.5){0}
\rput(-0.5,2.5){2}
\rput(-0.5,4.5){4}
\rput(-0.5,6.5){6}
\psline{->}(-1.0,6)(-1.0,7.5)
\rput(-1.5,6.7){$f$}

\end{pspicture}

\caption{
{$\Ext_A(A(0),\mathbf{C}_\eta)$} }
\label{fig:A(0)C}
\end{figure}

\section{Multiplication by $h_1$ for $A(0)$-free $A$-modules}
\label{sec:CohomA2}

Recall that $A/\!/A(0)$ denotes the Hopf algebra quotient $A \otimes_{A(0)} \bM_2$.
There is a short exact sequence
\[0 \ltarr \bM_2 \xltarr{\varepsilon} A/\!/A(0) \ltarr I \ltarr 0\]
of $A$-modules,
where $\varepsilon$ is the augmentation and $I$
is the augmentation ideal.
This short exact sequence 
gives rise to a long exact sequence 
\[ \Ext_A^{s,f,w} \xrtarr{\varepsilon^*} \Ext_A^{s,f,w}(A/\!/A(0),\bM_2) \rtarr \Ext_A^{s,f,w}(I,\bM_2) \xrtarr{\partial} \Ext_A^{s-1,f+1,w}  \]
of modules over $\Ext_A$.
We have a change-of-rings isomorphism 
\[\Ext_A(A/\!/A(0),\bM_2) \iso \Ext_{A(0)}(\bM_2,\bM_2) \iso \bM_2[h_0],\]
so $\varepsilon^*$ 
is an isomorphism when $s=0$ and $\partial$ an isomorphism when $s>1$.
 
The bottom class of $I$ is $\Sq^2$, which occurs in bidegree $(2,1)$, so that we may write $I \iso \Sigma^{2,1}N$ for a connective $A$-module $N$. 
 It follows that for $s>0$, we have a commutative square 
\[\xymatrix{
 \Ext_A^{s,f,w} \ar[r]^{h_1} & \Ext_A^{s+1,f+1,w+1} \\
 \Ext_A^{s+1,f-1,w}(I,\bM_2) \ar[u]^\iso_\partial \ar[r]_{h_1} & \Ext_A^{s+2,f,w+1}(I,\bM_2) \ar[u]_\iso^\partial \\
 \Ext_A^{s-1,f-1,w-1}(N,\bM_2) \ar[u]^\iso \ar[r]_{h_1} & \Ext_A^{s,f,w}(N,\bM_2) \ar[u]_\iso \\
  } \]
of $\Ext_A$-module maps.

\begin{lemma}\label{lem:NFreeA(0)} 
The $A$-module $N$ is free as a left $A(0)$-module. 
\end{lemma}

\begin{pf}
This argument is identical to the classical case.

We use the admissible $\bM_2$-basis for $A$, which consists of monomials of the form
$\Sq^{r_1} \Sq^{r_2} \cdots \Sq^{r_n}$ such that $r_i \geq 2 r_{i+1}$.
Then $A/\!/A(0)$ has an $\bM_2$-basis consisting of the admissible monomials
$\Sq^{r_1} \Sq^{r_2} \cdots \Sq^{r_n}$ such that $r_n > 1$.

When $r_1$ is even, $\Sq^1 \cdot \Sq^{r_1} \Sq^{r_2} \cdots \Sq^{r_n}$ 
equals $\Sq^{r_1+1} \Sq^{r_2} \cdots \Sq^{r_n}$.  Therefore,
the augmentation ideal $I$ of $A/\!/A(0)$ is a free left $A(0)$-module with basis
consisting of admissible monomials of the form
$\Sq^{r_1} \Sq^{r_2} \cdots \Sq^{r_n}$ such that $r_1$ is even and $r_n > 1$.
\end{pf}

\begin{prop}\label{MainProp}
Let $M$ be an $A$-module that is free as an $A(0)$-module and concentrated in nonnegative degrees. 
Then the map
\[ h_1:\Ext_A^{s,f,w}(M,\bM_2) \xrightarrow{} \Ext_A^{s+1,f+1,w+1}(M,\bM_2)\]
is an isomorphism if $ s < 2f-2$, and it is a surjection if $s\leq 2f$.
\end{prop}

We will mimic the classical argument of Adams \cite{A2}, with some variations to account
for motivic phenomena.

\begin{proof}
By Lemma \ref{lem:RestateLoc}, it suffices to show that
$\Ext_A^{s,f,w}(M, \mathbf{C}_\eta)$ vanishes when $s < 2f$.

We begin by recalling from Proposition \ref{prop:low-Ext} that
$\Ext_A^{s,f,w}(A(0), \mathbf{C}_\eta)$ vanishes when
$s < 2f$ and $f \leq 4$.

An arbitrary module $M$ can be built up iteratively as an extension
\[ 0 \rtarr \bigoplus\Sigma^{t_i,w_i} A(0) \rtarr M \rtarr M' \rtarr 0\]
of $A$-modules,
where 
$\Ext_A^{s,f,w}(M', \mathbf{C}_\eta)$ vanishes for $s < 2f$ and $f \leq 4$
by induction.
The long exact sequence in 
$\Ext$ then shows that 
$\Ext_A^{s,f,w}(M, \mathbf{C}_\eta)$ also vanishes for $s < 2f$ and $f \leq 4$.

We have now established the proposition for $f \leq 4$.  The next step is to extend
the result to larger values of $f$.  As in the previous step, 
we start with the special case $M = A(0)$.

In order to compute 
$\Ext_A^{s,f,w}(A(0), \mathbf{C}_\eta)$, we must construct a resolution 
for $A(0)$.
Proposition \ref{prop:Ceta-proj} says that this resolution can be built from
copies of $A$ or from copies of $\widetilde{A}$.

We construct a resolution 
\[ A(0) \leftarrow R_0 \leftarrow R_1\leftarrow R_2\leftarrow R_3\leftarrow \cdots\]
of $A(0)$ in the usual way by adding a copy of $A$ to $R_{n+1}$ for each indecomposable
in the kernel $K_n$ of the boundary map $R_n \rtarr R_{n-1}$.
However, when we find two indecomposable elements $x$ and $y$ of $K_n$ such that
$\Sq^2 x = \tau y$, we add one copy of $\widetilde{A}$ to $R_{n+1}$ rather than
two copies of $A$.

As a result of this process, one can verify that the kernel 
$K_3$ of the boundary map $R_3 \rtarr R_2$ vanishes in degrees less than $12$.
So we can write $K_3 = \Sigma^{12,0}D$, where $D$ is concentrated in non-negative 
degrees. 
Note that $D$ is $\bM_2$-free since it is a submodule of 
the $\bM_2$-free module $\Sigma^{-12,0} R_3$. 

Proposition~\ref{prop:AtildeA0Free} implies that each $R_n$ is $A(0)$-free.
Then Corollary~\ref{cor:A0FreeSOS} implies that $D$ is $A(0)$-free as well.

We already know from an earlier step that
$\Ext_A^{s,f,w}(D,\mathbf{C}_\eta)$ vanishes for $s < 2f$ and $f \leq 4$.
We have isomorphisms
\[ \Ext_A^{s,f,w}(D,\mathbf{C}_\eta) \iso \Ext_A^{s+12,f,w}(K_3,\mathbf{C}_\eta) \iso \Ext_A^{s+8,f+4,w}(A(0),\mathbf{C}_\eta)\]
for $f>0$, so 
it follows that 
$\Ext_A^{s,f,w}(A(0),\mathbf{C}_\eta)$ vanishes for $s < 2f$ and $f \leq 8$.

As before, we can then show that 
when $M$ is an arbitrary module,
$\Ext_A^{s,f,w}(M,\mathbf{C}_\eta)$ vanishes for $s < 2f$ and $f \leq 8$.
This establishes the proposition for $f \leq 8$.

We can repeat this process to establish the proposition for all $f$.
\end{proof}

\end{document}